\documentclass[11pt,a4]{amsart}
\makeatletter
\renewcommand{\section}{\@startsection{section}{1}{0pt}{20pt}{6pt}{\large\bfseries}}
\makeatother

\usepackage{enumerate}
\usepackage{pdfsync}

\numberwithin{equation}{section}

\theoremstyle{plain}
  \newtheorem{thm}{Theorem}[section]
  \newtheorem{prop}[thm]{Proposition}
  
   \newtheorem{cor}[thm]{Corollary}
\newtheorem{definition}[thm]{Definition}
  \newtheorem{remark}[thm]{Remark}
\newtheorem{exa}[thm]{Example}

\newcommand{\R}{\mathbb{R}}
\renewcommand{\Re}{{\mathfrak{Re}}}
\newcommand{\C}{\mathbb{C}}
\newcommand{\Q}{\mathbb{Q}}
\newcommand{\E}{\mathbb{E}}

\renewcommand{\P}{{\mathbb{P}}}

\newcommand{\Ex}{\Sigma}
\newcommand{\iex}{A}

\newcommand{\I}{\mathcal{I}}

\newcommand{\Id}[1]{{{\mathbb{I}}}_{\{#1\}}}

\begin{document}

\bibliographystyle{plain}

\title{A transformation for L\'evy processes with one-sided jumps and applications}

\author{M. Chazal}
   \address{Universit\'e Libre de Bruxelles,
Department of Mathematics,
Actuarial Sciences - CP210,
Boulevard du Triomphe,
B-1050 Bruxelles,
Belgium }
\email{mchazal@ulb.ac.be}

\author{A. Kyprianou}
\address{Department of Mathematical Sciences
University of Bath, Bath BA2 7AY, UK} \email{a.kyprianou@bath.ac.uk}

\author{P. Patie}
   \address{Universit\'e Libre de Bruxelles,
Department of Mathematics,
Actuarial Sciences - CP210,
Boulevard du Triomphe,
B-1050 Bruxelles,
Belgium }
\email{ppatie@ulb.ac.be}

\begin{abstract}
The aim of this work is to extend and study a family of transformations between Laplace exponents of L\'evy processes which have been introduced recently in a variety of different contexts, \cite{Patie-06-poch, Patie-08-ref-exp, KP, Gnedin}, as well as in older work of Urbanik \cite{Urbanik}. We show how some specific instances of this mapping prove to be useful for
a variety of applications.

\bigskip


\bigskip


\end{abstract}

\date{}

\maketitle

\section{Introduction}
In this paper we are interested in a
  L\'evy process with no positive jumps, possibly independently killed at a constant rate, henceforth denoted by  $\xi = (\xi_t, t\geq 0)$ with law $\mathbb{P}$.
That is to say, under $\mathbb{P}$, $\xi$ is a stochastic process which has almost surely c\`adl\`ag paths, increments that are stationary and independent and killed at an independent rate $\kappa\geq 0$. The case that $\kappa=0$ corresponds to no killing.
Whilst it is normal to characterise L\'evy processes by their characteristic exponent, in the case that the jumps are non-positive one may also comfortably work with the Laplace exponent which satisfies,
\[
\mathbb{E}(e^{u \xi_t}) = e^{\psi(u)t}, \, t\geq 0,
\]
where $u\geq 0$.
Indeed, it is a well established fact that the latter Laplace exponent is  strictly convex on $[0,\infty)$, see for example Bertoin \cite{Bertoin-96}. Moreover, the L\'evy-Khintchine formula necessarily allows us to write $\psi$ in the form
\begin{equation}
\psi(u) =-\kappa+au +\frac{1}{2}\sigma^{2}u^{2}+\int_{(-\infty, 0)}(\mathrm{e}^{u x}-1-u x\mathbf{1}_{(|x|<1)})\Pi ({\rm d}x),
\label{LLK}
\end{equation}
for $u\geq 0$ where $\kappa\geq 0$, $a\in\mathbb{R}$, $\sigma^2\geq 0$ and $\Pi$ is a measure concentrated on $(-\infty,0)$ satisfying $\int_{(-\infty, 0)} (1\wedge x^2) \Pi({\rm d}x)<\infty$.
Note in particular that our definition includes the case that $-\xi$ is a (possibly killed) subordinator. Indeed, when $\Pi$ satisfies $\int_{(-\infty, 0)} (1\wedge |x|)\Pi({\rm d}x)<\infty$ and we choose $\sigma=0$ and $a = -\mathtt{d} +\int_{(-1,0)} x\Pi({\rm d}x)$
we may write for $u\geq 0$,
\[
\psi(u) = -\kappa -{\mathtt d}u - \int_{(0,\infty)} (1 - {\rm e}^{-u x})\nu({\rm d} x),
\]
where $\nu(x,\infty ) = \Pi(-\infty, -x)$. When $\mathtt d\geq 0$, writing $S_t = - \xi_t$ for $t\geq 0$, $\mathtt d$ and  $\nu$ should be thought of as the drift and L\'evy measure of the subordinator $S  = (S_t, t\geq 0)$ respectively. Moreover, writing $\phi(u) = -\psi (u)$ for $u\geq 0$, we may think of $\phi$ as the Laplace exponent of $S$ in the classical sense, namely
\[
 \mathbb{E}(e^{-u S_t}) = e^{-\phi(u)t}, \, t\geq 0.
\]
In general we shall refer to $\Psi_\ell$ as the family of Laplace exponents of (possibly killed) L\'evy processes with no positive jumps which are killed at rate $\kappa\geq 0$ and are well defined on $(\ell,\infty)$ for $\ell \leq 0$.

Our main objective  is to introduce a parametric family of linear transformations which serves as a mapping from the  space of Laplace exponents of L\'evy processes with no positive jumps into itself and therewith explore how a family of existing results for L\'evy processes may be extrapolated further. The paper is structured as follows. In the next section we introduce our three-parameters transformation and derive some basic properties. We also describe its connection with some transformations which have already appeared in the literature. The remaining part of the paper deals with the applications of our transformation to different important issues arising in the framework of L\'evy processes and related processes. More specifically, in the third section we provide some ways for getting new expressions for the so-called scale function of L\'evy processes. Section 4 is devoted to the exponential functional of L\'evy processes and finally in the last section we develop some applications to
the study of positive self-similar Markov processes.

\section{The transformation $\mathcal{T}_{\delta,\beta}$}\label{introsection}

We begin with the definition of our new transformation and consider its properties as a mapping on $\Psi_\ell$.

\begin{definition} Suppose that $\psi\in\Psi_\ell$ where $\ell\leq 0$. Then for $\delta,\beta \geq 0$, with the additional constraint that $\psi'(0+) = \psi(0) = 0$ if $\beta = 0$, let
\begin{eqnarray*}
\mathcal{T}_{\delta,\beta}\psi(u) &:=& \frac{u+\beta-\delta}{u+\beta}
\psi(u+\beta)-\frac{\beta-\delta}{\beta}\psi(\beta), \quad u\geq -\beta.
\end{eqnarray*}
%In the special case that $\delta=\beta$ we write $\mathcal{T}_{\beta}$ in place of $\mathcal{T}_{\beta,\beta}$.
\end{definition}
%\noindent Next we extend the definition of $\mathcal{T}_{\delta,\beta}$ and introduce a third parameter.

%\begin{definition}
%Suppose that $\psi\in\Psi_\alpha$ for $\alpha\leq 0$. Then for any $\gamma,\delta,\beta$ such that $\delta\geq 0$ and $\gamma +\beta\geq \alpha$,
%\[
%\mathcal{T}^{\gamma}_{\delta,\beta}\psi (u)= \mathcal{T}_\gamma\circ\mathcal{T}_{\delta,\beta}\psi(u)
%= \frac{u}{u+\gamma}\left(\frac{u+\gamma+\beta-\delta}{u+\gamma+\beta}
%\psi(u+\gamma+\beta)-\frac{\beta-\delta}{\beta}\psi(\beta)\right),
%\]
%where $u\geq -(\gamma+\beta)$.
%\end{definition}

\noindent Let us make some immediate observations on the above definition. Firstly note that $\mathcal{T}_{\delta,\beta}$ is a linear transform.
 In the special case that $\delta = \beta$ we shall write $\mathcal{T}_{\beta}$ in place of  $\mathcal{T}_{\beta,\beta}$. The transform $\mathcal{T}_\beta$ was considered recently for general spectrally negative L\'evy processes in Kyprianou and Patie \cite{KP} and for subordinators (as a result of a path transformation known as sliced splitting) in Gnedin \cite{Gnedin}.
Next note that, for $\beta, \gamma$ such that $\beta+\gamma\geq 0$,
$$\mathcal{T}_{\gamma}\circ\mathcal{T}_{\beta} = \mathcal{T}_{\gamma+\beta}.$$
Note that one can in principle extend the definition of $\mathcal{T}_{\delta,\beta}$ to L\'evy processes which have jumps on both sides which, for example, have all their exponential moments, by writing such processes as the difference of spectrally negative L\'evy processes and applying the $\mathcal{T}_{\delta,\beta}$ linearly.

In the special case that $\delta = 0$ and $\beta\geq 0$
 we have $\mathcal{E}_{\beta}:= \mathcal{T}_{0,\beta}$ satisfies
 \[
 \mathcal{E}_\beta \psi(u)= \psi(u+\beta)-\psi(\beta), \, u\geq -\beta,
 \]
 where, as usual,  $\psi\in\Psi_\ell$.
 This is the classical Esscher transform for L\'evy processes with no positive jumps expressed in terms of Laplace exponents.
It will be convenient to note for later that if $\Phi(u):=\psi(u)/u$ then we may write
 \[
 \mathcal{T}_{\delta, \beta} \psi (u) = \mathcal{E}_\beta \psi(u) - \delta\mathcal{E}_\beta\Phi(u).
 \]
 In particular we see that when $\beta = 0$, the assumption that $\psi'(0+) = 0$ allows us to talk safely about $\Phi(0+)$.

 \bigskip

One may think of $\mathcal{T}_{\delta, \beta}$ as one of the many possible generalisations of the Esscher transform. For $\beta\geq 0$, the latter is a well-known linear transformation which maps $\Psi_\ell$ into itself and
 has proved to be a very effective tool in analysing many different fluctuation identities for L\'evy processes with no positive jumps. It is natural to ask if $\mathcal{T}_{\delta,\beta}$ is equally useful in this respect. A first step in answering this question is to first prove that $\mathcal{T}_{\delta,\beta}$ also maps $\Psi_\ell$ into itself. This has already been done for the specific family of  transformations $\mathcal{T}_\beta$ in Lemma 2.1 of \cite{KP}.

 \begin{prop} \label{prop:mapping}
 Suppose that $\psi\in\Psi_\ell$ where $\ell\leq 0$.
Fix $\delta, \beta\geq 0$   with the additional constraint that $\psi'(0+)=\psi(0)=0$ if $\beta = 0$.
Then $\mathcal{T}_{\delta,\beta}\psi\in\Psi_{\ell-\beta}\subseteq \Psi_\ell$ and has no killing component.
Moreover, if $\psi$ has Gaussian coefficient $\sigma$ and jump measure $\Pi$ then $\mathcal{T}_{\delta,\beta}\psi$ also has Gaussian coefficient $\sigma$ and its L\'evy  measure is given by
\[
e^{\beta x}\Pi({\rm d}x)+\delta e^{\beta x}\overline{\Pi}(x){\rm d}x + \delta\frac{\kappa}{\beta}e^{\beta x}{\rm d}x\,\, \text{ on }(-\infty,0),
\]
where $\overline{\Pi }(x) = \Pi(-\infty, -x)$ and we understand the final term to be zero whenever $\kappa =0$.
\end{prop}

\begin{proof}
Recall from earlier that  $ \mathcal{T}_{\delta, \beta} \psi (u) = \mathcal{E}_\beta \psi(u) - \delta\mathcal{E}_\beta\phi(u)$.  Moreover, from its definition, it is clear that argument of $\mathcal{T}_{\delta, \beta}$ may be taken for all $u$ such that $u+\beta \geq \ell$.  It is well understood that $\mathcal{E}_\beta\psi$ is the Laplace exponent of a spectrally negative L\'evy process without killing whose Gaussian coefficient remains unchanged but whose L\'evy measure is transformed from $\Pi({\rm d}x)$ to $e^{\beta x}\Pi({\rm d}x)$. See for example Chapter 8 of \cite{Kyprianou-06}.
The proof thus boils down to understanding the contribution from $-\delta\mathcal{E}_\beta\phi(u)$.

A lengthy but straightforward computation based on integration by parts shows that
\[
\Phi(u) = -\frac{\kappa}{u} + (a - \overline\Pi(-1)) + \frac{1}{2}\sigma^2 u + \int_{-\infty}^0 (\mathbf{1}_{(|x|<1)} - e^{xu})\overline{\Pi} (x){\rm d}x.\\
\]
From this it follows that
\begin{eqnarray*}
-\delta\mathcal{E}_\beta\phi(u) &=& -\delta\frac{\kappa}{\beta} \frac{u}{u+\beta} -\delta (a - \overline{\Pi}(1))u -\delta\int_{-\infty}^0 (1 - e^{x u})e^{\beta x}\overline{\Pi}(x){\rm d}x\\
&=& -\delta (a - \overline{\Pi}(1))u  -\delta\int_{-\infty}^0 (1 - e^{x u})e^{\beta x}\overline{\Pi}(x){\rm d}x
-\delta\frac{\kappa}{\beta}\int_0^\infty (1 - e^{-ux})e^{-\beta x}{\rm d}x.
\end{eqnarray*}
Here we understand the final integral above to be zero if $\kappa = 0$.
In that case we see that $-\delta \mathcal{E}_\beta\psi(u)$ is the Laplace exponent of a spectrally negative L\'evy process which has no Gaussian component and a jump component which is that of a negative subordinator with jump measure given by $\delta e^{\beta x}\overline{\Pi}(x){\rm d}x + \beta^{-1}\kappa\delta e^{\beta x}{\rm d}x$ on $(-\infty,0)$.
\end{proof}

\begin{remark}\rm
As $\mathcal{T}_{\delta, \beta}$ maps $\Psi_{\ell}$ in to itself, it is possible to iterate the application of this transformation to produce yet further examples of spectrally negative L\'evy process exponents.
Consider for example the transformation $\mathcal{T}^\gamma_{\delta, \beta} = \mathcal{T}_\gamma\circ\mathcal{T}_{\delta, \beta}$ where $\delta,\beta\geq 0$, $\gamma> 0$. It is straightforward to check for this particular example that
 \begin{eqnarray*}
\mathcal{T}^{\gamma}_{\delta,\beta}\psi(u) &=&\frac{u}{u+\gamma}\left(\frac{u+\gamma+\beta-\delta}{u+\gamma+\beta}
\psi(u+\gamma+\beta)-\frac{\beta-\delta}{\beta}\psi(\beta)\right).
\end{eqnarray*}
\end{remark}

\bigskip

Whilst it is now clear that the mappings $\mathcal{T}^\gamma_{\delta,\beta}$ may serve as a way of generating new examples of L\'evy processes with no positive jumps from existing ones, our interest is largely motivated by how the aforesaid transformation interacts with certain path transformations and fluctuation identities associated to L\'evy processes. Indeed, as alluded to above, starting with Urbanik \cite{Urbanik}, the formalisation of these transformations is motivated by the appearance of particular examples in a number of such contexts. On account of the diversity of these examples, it is worth recalling them here briefly for interest.

\bigskip

\noindent {\bf Applications to pssMp:} In Kyprianou and Patie \cite{KP} the transformation $\mathcal{T}_\beta$ turns out to give a natural encoding for the Ciesielski-Taylor identity as seen from the general point of view of spectrally negative positive self similar Markov processes (pssMp). Recall that Lamperti \cite{Lamperti-72} showed that for any $x>0$, there exists a one to one mapping between $\P_{\log x}$, the law of a L\'evy process killed at rate $\kappa\geq0$, starting from $\log x$, and the law
$\Q_{x}$ of a $\alpha$-self-similar positive Markov process $X$ on
$(0,\infty)$ where $\alpha\in\mathbb{R}$, i.e.~a Feller process which satisfies for any $c>0$,
\begin{equation} \label{eq:self}
\left((X_{t}, t\geq0), \Q_{cx}\right)
\stackrel{(d)}{=}\left((cX_{c^{-\alpha}t}, t\geq0), \Q_{x}\right).
\end{equation}
More specifically, Lamperti proved that $X$ can be constructed from
$\xi$ as follows
\begin{equation} \label{eq:lamp_transf}
\log\left(X_t\right) =  \xi_{\iex_t}, \: t\geq 0,
\end{equation}
 where
\begin{equation*}
 \iex_t = \inf
\{ s \geq 0; \: \Ex_s := \int_0^s e^{\alpha \xi_u} \: du
> t \}.
\end{equation*}  %For any $x>0$, we denote by ${\rm{E}}_x$  the expectation operator associated with $\Q_x$.

In the case that $\xi$ is spectrally negative with Laplace exponent $\psi$, we shall denote the law of $X$ by $\mathbb{Q}_x^\psi$. Under reasonably broad conditions it is well understood how to extend the definition of $X$ as  Markov process to include $0$ into its state space; that is to say, how to give meaning to the law $\mathbb{Q}_0$. This is done either by constructing an entrance  law for $X$ at $0$,  or by constructing a recurrent extension from $0$; see for example \cite{Bertoin-Yor-02-b, Caballero-Chaumont, Rivero-05, Fitz}. Within the class of pssMps which are spectrally negative, the construction of a unique entrance law is always possible whenever $\psi'(0+)\geq 0$ and  $\kappa =0$. The existence of a unique recurrent extension when $\psi'(0+)<0$ or $\kappa>0$ requires that $\theta: = \sup\{\lambda \geq 0 : \psi(\lambda)  =0\}\in(0,\alpha)$.

\medskip

{\it Ciesielski-Taylor identity:} Within this class of spectrally negative pssMps and in the spirit of the original Ciesielski-Taylor identity for Bessel processes (which themselves are included in the aforementioned category), \cite{KP} showed that for each $a>0$,
\[
(T_a, \mathbb{Q}^\psi_0) \stackrel{(d)}{=} \left(\int_0^\infty \mathbf{1}_{\{X_s\leq a\}}ds, \mathbb{Q}_0^{\mathcal{T}_\alpha\psi}\right),
\]
where $T_a = \inf\{t>0 : X_t =a\}$.

\medskip

{\it Continuous state branching processes with immigration:} Here, we take  $X$ to be  a self-similar continuous state branching process with immigration. We refer to Patie \cite[Section 4]{Patie-06-poch} for more detailed information on this family of Markov processes. In particular, it is shown   that $X$ is self-similar of index $0<
\alpha\leq1$  and its branching mechanism is $\varphi(u)=-\frac{c}{\alpha}u^{\alpha+1}, c>0,$ while its immigration mechanism is given by
$\chi(u)=c\delta(\frac{\alpha+1}{\alpha})u^{\alpha},\: \delta>0$. The L\'evy process $\xi$ associated to $X$ via the Lamperti mapping \eqref{eq:lamp_transf} is characterized by the Laplace exponent $\psi^{(\alpha,\delta)}$ of $-\xi$ which admits the following form
\begin{eqnarray*}
\psi^{(\alpha,\delta)}(u)&=& \mathcal{T}_{(\alpha+1)\delta,\alpha}\psi(u)
\end{eqnarray*}
where $\psi(u)=c\frac{\Gamma(u+1)}{\Gamma(u-\alpha)}$ is the Laplace exponent of the spectrally negative L\'evy process associated, via the Lamperti mapping, to the spectrally negative $(\alpha+1)-$stable L\'evy process killed upon entering into the negative real line, see \cite[Proposition 4.11]{Patie-06-poch}.
\bigskip

\noindent{\bf Sliced splitting of subordinators:} In combination with the study of certain combinatorial structures, Gnedin \cite{Gnedin} introduced the following method of sliced splitting the path of subordinators which corresponds to the application of a special case of the transformation we introduce here.

For a subordinator (no killing) with Laplace exponent $\phi$, run its path until it first crosses a level $\mathbf{e}$, which is independent and exponentially distributed with parameter $\beta$. Remove the path of the subordinator which lies above this level and issue an independent copy of the original subordinator from the crossing time at the spatial position $\mathbf{e}$. Repeat this procedure independently within the space-time framework of the latter subordinator and so on.

 To describe  more precisely the resulting process, suppose that $\{\mathbf{e}_i: i\geq 1\}$ are the inter-arrival times of arrival of a Poisson process with rate $\beta$ and suppose that $\{S^{(i)}:i\geq 1\}$ are independent copies of the subordinator whose Laplace exponent is $\phi$. Let $\tau_0 =0$ and for $i\geq 1$, $\tau_i = \inf\{t>0 : S^{(i)}_t > \mathbf{e}_i \}$. The path which was earlier described in a heuristic way is defined iteratively by $S_0 = 0$, and, for $i\geq 1$,
\[
S_t = S_{T_{i-1}}+ \min\{S^{(i)}_{t-T_{i-1}}, \mathbf{e}_i  \}\text{ when  }t\in(T_{i-1}, T_i],
\]
where $T_0 =0$ and, for $i\geq 1$, $T_i = \tau_1 + \cdots + \tau_i$.
Gnedin \cite{Gnedin} shows that the resulting process $S = (S_t , t\geq 0)$ is again a subordinator and moreover its Laplace exponent, $\phi_\beta$, satisfies
\[
\phi_\beta(u) =\mathcal{T}_\beta \phi(u)
\]
for $u\geq0$.

%\bigskip

\section{Scale functions for spectrally negative L\'evy processes}

Scale functions have occupied a central role in the theory of spectrally negative L\'evy processes over the last ten years. They appear naturally in virtually all fluctuation identities of the latter class and consequently have also been instrumental in solving a number of problems from within classical applied probability. See Kyprianou \cite{Kyprianou-06} for an account of some of these applications. Despite the fundamental nature of scale functions in these settings, until recently very few explicit examples of scale functions have been found. However in the recent work of Hubalek and Kyprianou \cite{Hubalek-Kyprianou-10}, Chaumont et al. \cite{Chaumont-Kyprianou-Pardo-09}, Patie \cite{Patie-06-poch}, Kyprianou and Rivero \cite{Kyprianou-Rivero-08}, many new examples as well as general methods for constructing explicit examples have been uncovered. We add to this list of contemporary literature by showing that the
 some of transformations introduced in this paper can be used to construct  new families of scale functions from existing examples.

Henceforth we shall assume that the underlying L\'evy process, $\xi$, is spectrally negative, but does not have monotone paths. Moreover, we allow, as above, the case of independent killing at rate $\kappa\geq 0$.  For a given spectrally negative L\'evy process with Laplace exponent $\psi$, its scale function $W_{\psi}:[0,\infty)\mapsto[0,\infty)$ is the unique continuous positive increasing
function
characterized by its Laplace transform as follows.  For any $\kappa\geq0$ and  $u>\theta:=\sup\{\lambda \geq 0 : \psi(\lambda ) = 0\}$,
\begin{eqnarray*}
 \int_0^{\infty}e^{-ux}W_{\psi}(x)dx = \frac{1}{\psi(u)}.
\end{eqnarray*}
In the case that $\kappa>0$, $W_\psi$ is also known as the $\kappa$-scale function.

Below we show how our new transformation generates new examples of scale functions from old ones; first in the form of a theorem and then with some examples.
%We have the
%following characterization of
%$W_{\mathcal{T}^{\beta}_{\delta,\theta}\psi}$.
\begin{thm}\label{scalefunctions}
Let $x,\beta\geq0$. Then,
\begin{eqnarray}
W_{\mathcal{T}_{\beta}\psi}(x)&=& e^{-\beta x}W_{\psi}(x)
+\beta \int_0^{x}e^{-\beta y}W_{\psi}(y)dy.
\label{scale1}
\end{eqnarray}
Moreover, if $\psi'(0+)\leq0$, then  for any $x,\beta,\delta \geq 0$, we have
\begin{eqnarray*}
W_{\mathcal{T}_{\delta,\theta}\psi}(x)&=&e^{-\theta
x}\left(W_{\psi}(x) +\delta e^{\delta
x}\int_0^{x}e^{-\delta
y}W_{\psi}(y)dy\right)
\end{eqnarray*}
\end{thm}
\begin{proof}
The first assertion is
proved by observing that
\begin{eqnarray}
 \int_0^{\infty}e^{-ux}W_{\mathcal{T}_{\beta}\psi}(x)dx &=& \frac{u+\beta}{u\psi(u+\beta)} \nonumber \\
&=& \frac{1}{\psi(u+\beta)} +\frac{\beta}{u\psi(u+\beta) },
\label{invertthis}
\end{eqnarray}
which agrees with the Laplace transform of the right hand side of (\ref{scale1}) for which an integration by parts is necessary. As scale functions are right continuous, the result follows by the uniqueness of Laplace transforms.

For the second claim,  first note that $\mathcal{T}_{\delta, \theta}\psi = (u+\theta-\delta)\psi(u+\theta)/(u+\theta)$. A straightforward calculation shows that
 for all $u+\delta>\theta$, we have
\begin{eqnarray} \label{eq:sc}
 \int_0^{\infty}e^{-ux}e^{(\theta  -\delta)x}W_{\mathcal{T}_{\delta,\theta}\psi}(x)dx &=& \frac{u+\delta}{u\psi(u+\delta)}.
 \end{eqnarray}
The result now follows from the first part of the theorem.
\end{proof}

When $\psi(0+)>0$ and $\psi(0)=0$, the first  identity in the above theorem contains part of the conclusion in Lemma 2 of  Kyprianou and Rivero \cite{Kyprianou-Rivero-08}.  However, unlike the aforementioned result, there are no further restrictions on the underlying L\'evy processes and the expression on the right hand side is written directly in terms of the scale function $W_\psi$ as opposed to elements related to the L\'evy triple of the underlying descending ladder height process of $\xi$.

Note also that in the case that $\psi$ is the Laplace exponent of an unbounded variation spectrally negative L\'evy process, it is known that  scale functions are almost everywhere differentiable and moreover that they are equal to zero at zero; cf. Chapter 8 of \cite{Kyprianou-06}. One may thus integrate by parts the  expressions in the theorem above and obtain the following slightly more compact forms,
\begin{eqnarray*}
W_{\mathcal{T}_{\beta}\psi}(x)&=& \int_0^{x}e^{-\beta y}W'_{\psi}(y)dy\, \text{ and }\,
W_{\mathcal{T}_{\delta,\theta}\psi}(x)=  e^{-(\theta-\delta)
x}\int_0^{x}e^{-\delta y}W'_{\psi}(y)dy.
\end{eqnarray*}

We conclude this section by giving some examples.

\begin{exa}[(Tempered) Stable Processes]\rm \label{ex:sp}
Let $\psi_{\kappa,c}(u)=(u+c)^{\alpha} - c^\alpha- \kappa$ where $1<\alpha<2$ and $\kappa,c\geq 0$.  This is the Laplace exponent of an unbounded variation tempered stable spectrally negative L\'evy process $\xi$ killed at an independent and exponentially distributed time with rate $\kappa$. In the case that $c=0$, the underlying L\'evy process is just  a regular
spectrally negative $\alpha$-stable L\'evy process. In that case it is known that
\[
\int_0^\infty e^{-ux} x^{\alpha-1}\mathcal{E}_{\alpha,\alpha}(\kappa x^{\alpha})dx = \frac{1}{u^\alpha-\kappa}
\]
and hence the scale function is given by
\begin{eqnarray*}
W_{\psi_{\kappa,0}}(x)&=& x^{\alpha-1}\mathcal{E}_{\alpha,\alpha}(\kappa x^{\alpha}),\quad x\geq 0,
\end{eqnarray*}
where $\mathcal{E}_{\alpha,\beta}(x)=\sum_{n=0}^{\infty}\frac{x^n}{\Gamma(\alpha n+\beta)}$ stands for the generalized Mittag-Leffler function. (Note in particular that when $\kappa=0$ the expression for the scale function simplifies to $\Gamma(\alpha)^{-1} x^{\alpha-1}$).
Since
\[
\int_0^\infty e^{- u x} e^{-cx}W_{\psi_{\kappa+ c^\alpha,0}}(x) dx = \frac{1}{(u+c)^{\alpha} - c^\alpha- \kappa}
\]
it follows that
\[
W_{\psi_{\kappa,c}}(x)= e^{-cx}W_{\psi_{\kappa+ c^\alpha,0}}(x) = e^{-cx}x^{\alpha-1}\mathcal{E}_{\alpha,\alpha}((\kappa+ c^\alpha)x^{\alpha})
\]
Appealing to the first part of Theorem \ref{scalefunctions} we now know that for $\beta \geq 0,$
\[
 W_{\mathcal{T}_{\beta}\psi_{\kappa,c}}(x) =
e^{-(\beta +c) x }x^{\alpha-1}\mathcal{E}_{\alpha,\alpha}((\kappa+ c^\alpha)x^{\alpha})
+ \beta \int_0^x e^{-(\beta +c) y }y^{\alpha-1}\mathcal{E}_{\alpha,\alpha}((\kappa+ c^\alpha)y^{\alpha})dy.
\]

Note that $\psi_{\kappa,c}'(0+) = \alpha c^{\alpha-1}$ which is zero if and only if $c=0$. We may use the second and third part of Theorem \ref{scalefunctions} in this case. Hence, for any $\delta>0$, the scale function of the spectrally negative L\'evy process with Laplace exponent
$\mathcal{T}_{\delta,0}\psi_{0,0}$ is
\begin{eqnarray*}
W_{\mathcal{T}_{\delta,0}\psi_{0,0}}(x)&=& \frac{1}{\Gamma(\alpha-1)} e^{\delta
x}\int_0^xe^{-\delta y}y^{\alpha-2} dy\\
&=&\frac{\delta^{\alpha-1}}{\Gamma(\alpha-1)} e^{\delta
x}\Gamma(\alpha-1,\delta x)
\end{eqnarray*}
where we have used the recurrence relation for the Gamma function and $\Gamma(a,b)$ stands for the incomplete Gamma function of parameters $a,b>0$. Moreover, we have, for any $\beta >0$,
\begin{eqnarray*}
W_{\mathcal{T}^{\beta}_{\delta,0}\psi_{0,0}}(x)&=& \frac{1}{\Gamma(\alpha-1)} \left( \frac{\beta^{\alpha} }{\beta-\delta}\Gamma(\alpha-1,\beta x)- e^{(\beta-\delta)
x}\frac{\delta^{\alpha} }{\beta-\delta}\Gamma(\alpha-1,\delta x)\right).
\end{eqnarray*}
Finally, the scale function of the spectrally negative L\'evy process with Laplace exponent
$\mathcal{T}^{\beta}_{\delta,0}\psi_{\kappa,0}$ is given by
\begin{eqnarray*}
W_{\mathcal{T}_{\delta,0}\psi_{\kappa,0}}(x)&=& (x/\delta)^{\alpha-1}\mathcal{E}_{\alpha,\alpha-1}\left(x;\frac{\kappa}{\delta}\right)\\
W_{\mathcal{T}^{\beta}_{\delta,0}\psi_{\kappa,0}}(x)&=& \frac{\beta}{\beta-\delta}(x/\beta)^{\alpha-1}\mathcal{E}_{\alpha,\alpha-1}\left(x;\frac{\kappa}{\beta}\right) - \frac{\delta}{\beta-\delta}e^{-(\beta-\delta)
x}(x/\delta)^{\alpha-1}\mathcal{E}_{\alpha,\alpha-1}\left(x;\frac{\kappa}{\delta}\right).
\end{eqnarray*}
where we have used the notation
\begin{eqnarray*}
\mathcal{E}_{\alpha,\beta}\left(x;\kappa\right)=\sum_{n=0}^{\infty}\frac{\Gamma(x;\alpha n+\beta) \kappa^n}{\Gamma(\alpha n+\beta)}.
\end{eqnarray*}
\end{exa}

\begin{exa}[$\mathcal{T}$-stable process]\rm
Let us consider, for $1<\alpha<2$, $ \psi(u) =
(u-1)_{\alpha}$, $u \geq 0$, where $(u)_\alpha = \frac{\Gamma(u+\alpha)}{\Gamma(u)}$ stands for the Pochhammer symbol, and note that $\psi(1)=0$. It was shown in \cite{Chaumont-Kyprianou-Pardo-09, Patie-06-poch} that its scale function  is given
by
\begin{equation*}
\mathcal{W}_{\psi}(x)=\frac{1}{ \Gamma(\alpha)}e^{-
x}(1-e^{-x})^{\alpha-1},\quad x\geq0.
\end{equation*}
Hence, the scale function of the spectrally negative L\'evy process with Laplace exponent
$\mathcal{T}_{\delta,1}\psi$ is
\begin{eqnarray*}
W_{\mathcal{T}_{\delta,1}\psi}(x)= \frac{e^{(\delta-1)x}}{\Gamma(\alpha)}
\int_0^x e^{-(\alpha+\delta) y}(e^{y}-1)^{\alpha-2}(\alpha-e^y)dy
\end{eqnarray*}
and
\begin{eqnarray*}
W_{\mathcal{T}^{\beta}_{\delta,1}\psi}(x)= &=& \frac{\beta}{\beta+1-\delta}\int_0^{x}e^{-(\beta+1+\alpha)
y}(e^{y}-1)^{\alpha-2}(\alpha-e^y)dy \\&+& \frac{1-\delta}{\beta+1-\delta}e^{-(\beta+1-\delta)
x}\int_0^{x}e^{-(\delta+\alpha)
y}(e^{y}-1)^{\alpha-2}(\alpha-e^y)dy.
\end{eqnarray*}

\end{exa}

%%%%%%%%%%%%%%%%%%%%%%%%%
%%%%%%%%%%%%%%%%%%%%%%%%%
%%%%%%%%%%%%%%%%%%%%%%%%%
%%%%%%%%%%%%%%%%%%%%%%%%%
%%%%%%%%%%%%%%%%%%%%%%%%%
%%%%%%%%%%%%%%%%%%%%%%%%%
%%%%%%%%%%%%%%%%%%%%%%%%%
%%%%%%%%%%%%%%%%%%%%%%%%%

\section{Exponential functional and length-biased distribution}
In this part, we aim to study the effect of the transformation to the law of the exponential functional of some L\'evy processes, namely for subordinators and spectrally negative L\'evy processes. We recall that this random variable is defined by
\begin{eqnarray*}
I_{\psi} =  \int_0^{\infty}e^{-\xi_s}ds.
\end{eqnarray*}
Note that $ \lim_{t\rightarrow \infty} \xi_t =+\infty\: a.s.\Leftrightarrow I_{\psi}<\infty$ a.s. which is equivalent, in the spectrally negative case, to $\E[\xi_1]=\psi'(0^+)>0$. We refer to the survey paper of Bertoin and Yor \cite{Bertoin-Yor-05} for further discussion on  this random variable.  We also mention that Patie  in \cite{Patie-06c}, \cite{Patie-09-cras} and \cite{Patie-abs-08}, provides some explicit characterizations of its law in the case $\xi$ is a spectrally positive L\'evy process. In the sequel, we denote  by $f_{\psi}(dx)$  the law of   $I_{\psi}$.
\subsection{The case of subordinators}
Let us first assume that $\xi$ is a subordinator, that is a non-negative valued L\'evy process. We recall  that, in this case, the Laplace exponent $\phi$  of $\xi$ is given by $\phi(u)=-\psi(u)$, $u\geq0$, where in particular we allow for the case of killing ($\kappa\geq0$). In that case we write $I_\phi$ in place of $I_\psi$ and if $\phi(0)=\kappa$, we have that
\[
I_{\phi} =  \int_0^{{\bf{e}}_\kappa}e^{-\xi_s}ds
\]
where  ${\bf{e}}_\kappa$ stands, throughout, for an exponential random variable of mean $\kappa^{-1}>0$,  independent of $\xi$ (we have ${\bf{e}}_0=\infty$).
Before stating our result, we  recall that Carmona et al.~\cite{Carmona-Petit-Yor-98} determine the law of $I_{\phi}$ through its positive entire moments as follows
\begin{eqnarray} \label{eq:mom_sub}
\E[I_{\phi}^n] &=& \frac{n!}{\prod_{k=1}^n\phi(k)},\: n=0,1\ldots.
\end{eqnarray}
\begin{thm} \label{thm:exp_sub}
For any $\kappa,\beta\geq 0$,  the following identity
\begin{eqnarray}
f_{\mathcal{T}_{\beta}\phi}(dx) &=& \frac{x^{\beta}f_{\phi}(dx)}{\E[I_{\phi}^\beta]},\quad x>0,
\end{eqnarray}
holds. %In particular, $I_{\mathcal{T}_{1}\phi}$  is the  length-biased version of the random variable $I_{\phi}$, i.e.~ for any $g \in \mathbb{B}(\R^+)$, we have
%\begin{eqnarray}
%\E[g(I_{\mathcal{T}_{1}\phi})] &=& \frac{1}{\E[I_{\phi}]}\E[I_{\phi}g(I_{\phi})].
%\end{eqnarray}

\end{thm}
%\begin{remark}\rm
%The semigroup property of $\mathcal{T}_{\beta}$ entails that for any $\beta\geq0$, $I_{\mathcal{T}_{\beta+1}\phi}$  is the length-biased version of the random variable $I_{\mathcal{T}_{\beta}\phi}$.
%\end{remark}

%\begin{remark}\rm
%An interesting instance of length-biased distributions is  when the two random variables are related by the identity
%$I_{\mathcal{T}_{1}\phi}=\kappa I_{\phi}$ for some $\kappa\in (0,1)$. Indeed, under such a condition, it is not difficult to check that the random variable shares the same entire moments than the lognormal distribution. Since the random variables studied in this part are moment determinate, such an identity is not possible. We refer to Example \ref{ex:pois} below for a  discussion on  this topic.
%\end{remark}

\begin{proof}
Carmona et al.~\cite{Carmona-Petit-Yor-98}, see also Maulik and Zwart \cite[Lemma 2.1]{Maulik-Zwart-06}, determine the law of
$I_{\phi}$ by computing its positive entire moments which they derive from the following recursive equation, for any $s,\beta>0$ and $\kappa\geq0$,
\begin{eqnarray} \label{eq:rec_sub}
\E[I_{\mathcal{T}_{\beta}\phi}^s] &=& \frac{\mathcal{T}_{\beta}\phi(s)}{s}\E[I_{\mathcal{T}_{\beta}\phi}^{s-1}]
\\
&=&\frac{\phi(s+\beta)}{s+\beta}\E[I_{\mathcal{T}_{\beta}\phi}^{s-1}] \nonumber.
\end{eqnarray}
On the other hand, we also have, for any $s,\beta>0$,
\begin{eqnarray*}
\E[I_{\phi}^{s+\beta}]
&=&\frac{\phi(s+\beta)}{s+\beta}\E[I_{\phi}^{s-1+\beta}].
\end{eqnarray*}
 We get the first assertion by uniqueness of the solution of the recursive equation and the fact that the law of $I_{\phi}$
  is moment determinate.
\end{proof}

Before providing some new examples, we note  from Theorem \ref{thm:exp_sub} that if $I_{\phi}\stackrel{(d)}{=}AB$ for some independent random variables $A,B$ then the positive entire moments of $I_{\mathcal{T}_{\beta}\phi},\: \beta>0$ admit the following expression
\begin{eqnarray} \label{eq:mom_fac}
\E[I_{\mathcal{T}_{\beta}\phi}^n] &=& \frac{\E[A^{n+\beta}]}{\E[A^{\beta}]}\frac{\E[B^{n+\beta}]}{\E[B^{\beta}]},\: n=0,1\ldots.
\end{eqnarray}

\begin{exa}[Poisson process]\rm \label{ex:pois}
Let $\xi$ be a  Poisson process with mean $c=-\log(q)>0$ with $0<q<1$, i.e.~$\phi(u)  =  -\log(q)(1-e^{-u}), u \geq 0$. Biane et al.~\cite{Bertoin-Biane-Yor-04} computed the law of $I_{\phi}$ by means of $q-$calculus. More precisely, they show that its law is self-decomposable and is  given by
\begin{eqnarray*}
f_{\phi}(dy)= \sum_{n=0}^{\infty} (-1)^n e^{-y/q^n}\frac{q^{\frac{n(n-1)}{2}}}{(q;q)_{\infty}(q;q)_n}dy, \: y\geq0,
\end{eqnarray*}
where
\begin{eqnarray*}
(a;q)_n = \prod_{k=1}^{n-1}(1-aq^j),\:(q;q)_{\infty}= \prod_{k=1}^{\infty}(1-aq^j)
\end{eqnarray*}
and its Mellin transform is given, for any $s>0$, by
\begin{eqnarray*}
\E[I_{\phi}^{s}]
&=&\frac{\Gamma(1+s)(q^{1+s};q)_{\infty}}{(q;q)_{\infty}}
\end{eqnarray*}
The image of $\xi$ by the mapping $\mathcal{T}_{\beta}$ is simply a compound Poisson process with parameter  $c$ and jumps which are exponentially distributed on $(0,1)$ with parameter $\beta$, i.e.~$\mathcal{T}_{\beta}\phi(u) = \frac{u}{u+\beta} (-\log(q)(1-e^{-(u+\beta)}))$. Thus, we obtain that the law of $I_{\mathcal{T}_{\beta}\phi}$ has an absolute continuous density given by
\begin{eqnarray*}
f_{\mathcal{T}_{\beta}\phi}(dx)= \frac{x^{\beta}}{\E[I_{\phi}^{\beta}]}\sum_{n=0}^{\infty} e^{-x/q^n}(-1)^n\frac{q^{\frac{n(n-1)}{2}}}{(q;q)_{\infty}(q;q)_n}dx, \: x>0.
\end{eqnarray*}
To conclude this example, we mention that Bertoin et al.~\cite{Bertoin-Biane-Yor-04} show that the distribution of the random variable $L_{\phi}$ defined, for any bounded Borel function $f$, by
\begin{eqnarray*}
\E[f(L_{\phi})]
&=&\frac{1}{\E[I^{-1}_{\phi}]}\E[I^{-1}_{\phi}f(I'_{\phi}I^{-1}_{\phi})]
\end{eqnarray*}
with $I'_{\phi}$ an independent copy of $I_{\phi}$,
shares the same moments than the log normal distribution. It is not difficult to check that such transformation applied to $I_{\mathcal{T}_{\beta}\phi}$ does not yield the same properties.
\end{exa}

\begin{exa}[Killed compound Poisson process with exponential jumps]\rm
Let $\xi$ be a compound Poisson process of parameter $c>0$  with exponential jumps of mean $b^{-1}>0$ and killed at a rate $\kappa\geq0$. Its Laplace exponent has the form $\phi(u)=c\frac{u}{u+b}+\kappa$ and  its L\'evy measure is given by $\nu(dr)=cbe^{-br}dr,\:r>0$. We obtain from \eqref{eq:mom_sub} that
\begin{eqnarray*}
\E[I_{\phi}^{n}]
&=&\frac{n!\Gamma(n+b+1)\Gamma(\kappa_b+1)}{((\kappa+c))^n\Gamma(b+1)\Gamma(n+\kappa_b+1)}
\end{eqnarray*}
where we have set $\kappa_b=\frac{\kappa}{\kappa+c}b$.
Then, noting that $b-\kappa_b>0$, we get the identity in distribution
\begin{eqnarray*}
I_{\phi}\stackrel{(d)}{=} ((\kappa+c))^{-1}{\bf{e}}_1 B(\kappa_b+1,b-\kappa_b)
\end{eqnarray*}
where $B(a,b)$ stands for a Beta random variable of parameters $a,b>0$ and the random variables on the right hand side are considered to be independent.  The case $\kappa=0$ was considered by Carmona et al.~\cite{Carmona-Petit-Yor-97}. Finally,  observing that, for any $\beta\geq0$,  $\mathcal{T}_{\beta}\phi(u)=c\frac{u}{u+b+\beta}+\kappa\frac{u}{u+\beta}$ and  its L\'evy measure is $\nu_{\beta}(dr)= e^{-\beta r}(c(b+\beta)e^{-br}+\kappa\beta)dr,\:r>0$, we deduce from Theorem \ref{thm:exp_sub}, the identity
\begin{eqnarray*}
I_{\mathcal{T}_{\beta}\phi}\stackrel{(d)}{=} (a(\kappa+1))^{-1}G(\beta+1) B(\kappa_b+\beta+1,b-\kappa_b)
\end{eqnarray*}
where $G(a)$ is an independent  Gamma random variable with parameter $a>0$.
\end{exa}

\begin{exa}[The $\alpha$-stable subordinator]\rm
Let us consider, for $0<\alpha<1$, $ \phi( u) =u^{\alpha}, \: u \geq 0$ and in this case $\nu(dr)=\frac{\alpha r^{-(\alpha+1)}}{\Gamma(1-\alpha)}dr, r>0$. The law of $I_{\phi}$ has been characterized by Carmona et al.~\cite[Example E and Proposition 3.4]{Carmona-Petit-Yor-97}. More precisely, they show that the random variable $Z= \log I_{\phi} $ is self-decomposable and admits the following L\'evy-Khintchine representation
\begin{eqnarray*}
\log \E[e^{iuZ}]&=&\log(\Gamma(1+iu))^{1-\alpha}\\
&=& (1-\alpha)\left(-iu E_{\gamma} + \int^{0}_{-\infty} (e^{ius}-ius-1) \frac{e^{s}}{|s|(1-e^s)}ds\right)
\end{eqnarray*}
where $E_{\gamma}$ denotes the Euler constant. First, note that  $\mathcal{T}_{\beta}\phi(u)=(u+\beta)^{\alpha}-\beta(u+\beta)^{\alpha-1}$ and $\nu_{\beta}(dr)= \frac{1}{\Gamma(1-\alpha)}e^{-\beta r}r^{-(\alpha+1)}(\alpha+\beta r)dr, r>0$.
Thus, writing $Z^{\beta}=\log( I_{\mathcal{T}_{\beta}\phi})$, we obtain
\begin{eqnarray*}
\log \E[e^{iuZ^{\beta}}]&=&\log\left(\frac{\Gamma(1+\beta + iu)}{\Gamma(1+\beta)}\right)^{1-\alpha}\\
&=& (1-\alpha)\left(iu \Upsilon(1+\beta) + \int^{0}_{-\infty} (e^{ius}-ius-1) \frac{e^{(1+\beta)s}}{|s|(1-e^s)}ds\right)
\end{eqnarray*}
where $\Upsilon$ stands for the digamma function, $\Upsilon(z)=\frac{\Gamma'(z)}{\Gamma(z)}$. Observing that $\lim_{\alpha \rightarrow 0} \mathcal{T}_{\beta}\phi(u) = \frac{u}{u+\beta}$, by passing to the limit in the previous identity we recover the previous example.

\end{exa}

\begin{exa}[The Lamperti-stable subordinator]\rm
Now let  $ \phi( u) =\phi_0(u)=
(\alpha u)_{\alpha},  u \geq 0$, with  $0<\alpha<1$. This example is treated by Bertoin and Yor in \cite{Bertoin-Yor-05}. They obtain
\begin{eqnarray*}
I_{\phi}\stackrel{(d)}{=} {\bf e}_1{\bf e}^{-\alpha}_1.
\end{eqnarray*}
Hence, with $\mathcal{T}_{ \beta}\phi(u)=\frac{u}{u+\beta}(\alpha (u+\beta))_{\alpha}$, we get
\begin{eqnarray*}
I_{\mathcal{T}_{\beta}\phi}\stackrel{(d)}{=} G(\beta+1)G(\beta+1)^{-\alpha}.
\end{eqnarray*}

\end{exa}

\subsection{The spectrally negative case}
Let us assume now that $\xi$ is a spectrally negative L\'evy process.  We recall that if $\E[\xi_1]<0$, then there exits $\theta>0$ such that
\begin{eqnarray}\label{eq:lap-levy_s}
\E[e^{\theta \xi_1}]=1
\end{eqnarray}
and we write $\psi_{\theta}(u)=\psi(u+\theta)$.   We proceed by  mentioning  that, in this setting, Bertoin and Yor \cite{Bertoin-Yor-02} determined the law of $I_{\psi}$  by computing its negative entire moments as
follows. If $\psi(0)=0$ and $\psi'(0^+)>0$, then, for any integer $n\geq1$,
\begin{eqnarray} \label{eq:m_sn}
\E[I_{\psi}^{-n}] &=& \psi'(0^+)\frac{\prod_{k=1}^{n-1}\psi(k)}{\Gamma(n)},
\end{eqnarray}
with the convention that $\E[I_{\psi}^{-1}] =
\psi'(0^+)$. Next, it is easily seen that  the strong Markov property for L\'evy processes yields, for any $a>0$,
\[I_{\psi} \stackrel{(d)}{=} \int_0^{T_a}e^{-\xi_s}ds + e^{-a}I'_{\psi}\]
where $T_a=\inf\{s>0; \: \xi_s\geq a\}$ and $I'_{\psi}$ is an independent copy of $I_{\psi}$, see e.g.~\cite{Rivero-05}. Consequently, $I_{\psi}$ is a positive self-decomposable random variable and thus its law is absolutely continuous with an unimodal density, see Sato \cite{Sato-99}. We  still denote its density by $f_{\psi}(x)$. Before stating our next result, we introduce the so-called Erd\'elyi-Kober operator of the first kind and  we refer to the monograph of Kilbas et al.~\cite{Kilbas-06} for background on  fractional operators. It is defined, for a smooth function $f,$ by
\begin{eqnarray*}
\mathbb{D}^{\alpha,\delta}f(x) &=& \frac{x^{-\alpha-\delta}}{\Gamma(\delta)}\int_0^xr^{\alpha}(x-r)^{\delta-1}f(r)dr,\quad x>0,
\end{eqnarray*}
where $\Re(\delta)>0$ and $\Re(\alpha)>0$. Note that this operator can be expressed in terms of the Markov kernel associated to a Beta random variable. Indeed, after performing a change of variable, we obtain
\begin{eqnarray*}
\mathbb{D}^{\alpha,\delta}f(x) &=& \frac{\Gamma(\alpha+1)}{\Gamma(\alpha+\delta+1)} \E\left[f\left(B(\alpha+1,\delta)x\right)\right]
\end{eqnarray*}
which motivates the following notation
\begin{eqnarray} \label{eq:defd}
\mathbf{D}^{\alpha,\delta}f(x) &=& \frac{\Gamma(\alpha+\delta+1)}{\Gamma(\alpha+1)}\mathbb{D}^{\alpha,\delta}f(x).
\end{eqnarray}

\begin{thm}
\begin{enumerate}
\item If $\psi'(0^+)>0$, then for any $\beta>0$, we have
\begin{eqnarray}
f_{\mathcal{T}_{\beta}\psi}(x) &=& \frac{x^{-\beta}f_{\psi}(x)}{\E[I_{\psi}^{-\beta}]},\: x>0.
\label{*}
\end{eqnarray}
In particular, $I_{\psi}$ is the length-biased distribution of $I_{\mathcal{T}_{1}\psi}$.
\item Assume that $\psi'(0^+)<0$. Then, for any $0<\delta<\theta$,  we have
\begin{eqnarray} \label{eq:f_ex}
I_{\mathcal{T}_{\delta,\theta}\psi} &\stackrel{(d)}{=}&
B^{-1}(\theta-\delta,\delta) I_{\psi_{\theta}}.
\end{eqnarray}
 This identity reads in terms of the Erd\'elyi-Kober operator as follows
 \begin{eqnarray*} \label{eq:d_ex}
f_{\mathcal{T}_{\delta,\theta}\psi}(x)=\mathbf{D}^{\theta-\delta-1,\delta}f_{\psi_{\theta}}(x),\: x>0.
\end{eqnarray*}
In particular, we have
\begin{eqnarray} \label{eq:d_a}
f_{\mathcal{T}_{\delta,\theta}\psi}(x) \sim \frac{ \Gamma(\theta)}{\Gamma(\delta)\Gamma(\theta-\delta)}\E[I^{\theta-\delta}_{\psi_{\theta}}] \: x^{\delta-\theta-1} \quad \textrm{as}  \quad x\rightarrow \infty,
\end{eqnarray}
($f(t)\sim g(t) $ as $t\rightarrow a$ means that $\lim_{t\rightarrow a}\frac{f(t)}{g(t)}=1$ for any $a \in [0,\infty]$).
Combining the two previous results, we obtain, for any $0<\delta<\theta$ and $\beta\geq 0$,
\begin{eqnarray*}
f_{\mathcal{T}^{\beta}_{\delta,\theta}\psi}(x) &=& \frac{x^{-\beta}}{\E[I_{\mathcal{T}_{\delta,\theta}\psi}^{-\beta}]}\mathbf{D}^{\theta-\delta-1,\delta}f_{\psi_{\theta}}(x),\: x>0
\end{eqnarray*}
and
\begin{eqnarray*}
f_{\mathcal{T}^{\beta}_{\delta,\theta}\psi}(x) \sim \frac{\Gamma(\theta)}{\Gamma(\theta-\delta)\Gamma(\delta)}\frac{
\E[I_{\psi_{\theta}}^{-\beta}]}{\E[I^{\theta-\delta}_{\psi_{\theta}}]}\: x^{\delta-\theta-\beta-1} \quad \textrm{as}  \quad x\rightarrow \infty .
\end{eqnarray*}
%\item Finally, if there exists $\gamma>0$ such that $\psi(-\gamma)=0$, then
%\begin{eqnarray*}
%f_{\psi}(x) \sim \frac{\Gamma(\theta)}{\Gamma(\theta-\delta)\Gamma(\delta)}\frac{
%\E[I_{\psi_{\theta}}^{-\beta}]}{\E[I^{\theta-\delta}_{\psi}]}\: x^{\gamma-1} \quad \textrm{as}  \quad x\rightarrow \infty .
%\end{eqnarray*}

\end{enumerate}
\end{thm}
\begin{remark}\rm
We point out, that under the Cr\'amer condition, Rivero \cite[Lemma 4]{Rivero-05} obtains the following large asymptotic for the tail of the distribution
\begin{equation} \label{eq:ar}
\P\left(I_{\Psi}>x\right) \sim C x^{-\theta} \quad \textrm{as}  \quad x\rightarrow \infty,
\end{equation}
where $C>0$ and $\Psi(\theta)=0$ with $\theta>0$ and $\Psi$ being here the Laplace exponent of a non-lattice L\'evy process. Rivero has a characterization
of the constant $C$ only in the case $\theta<1$. In our setting, it is plain that, for any $\beta\geq0$, $\mathcal{T}^{\beta}_{\delta,\theta}\psi(\delta-\theta-\beta)=0$. Moreover, $f_{\mathcal{T}^{\beta}_{\delta,\theta}\psi}(x)$ is unimodal and  hence ultimately monotone at infinity, therefore our asymptotic results, up to the constant characterization,  could have been also  deduced from \eqref{eq:ar}.
\end{remark}

\begin{proof}
We start by  recalling the following identity due to Bertoin and Yor \cite{Bertoin-96-b}
\begin{eqnarray*}
I_{\psi}/I_{\phi} &\stackrel{(d)}{=} &e^{-M} {\bf e}_1
\end{eqnarray*}
where $\phi(u)= \psi(u)/u, u\geq0$, that is $\phi$ is the Laplace exponent of the ladder height process of the dual L\'evy process,  $M=\sup_{t\geq0}\{-\xi_t\}$ is the overall maximum of the dual L\'evy process and the random variables are considered to be independent. Thus, recalling that $\E[e^{-sM}]=\psi'(0^+)/\phi(s)$, we have, for any $s>0$,
\begin{eqnarray}
\E[I_{\psi}^{-(s+1)}] &=& \frac{\psi'(0^+)\Gamma(s+2)}{\phi(s+1)\E[I_{\phi}^{s+1}]}
\nonumber \\
&=&\frac{\psi'(0^+)\Gamma(s+1)}{\E[I_{\phi}^{s}]} \nonumber \\
&= & \frac{\psi(s)}{s}\E[I_{\psi}^{-s}] \label{eq:rr}
\end{eqnarray}
where we have used the recurrence relationship satisfied by $I_{\phi}$, see \eqref{eq:rec_sub}. Similarly to the case of subordinators,  we have, for any $s>0$,
\begin{eqnarray*}
\E[I_{\mathcal{T}_{\beta}\psi}^{-(s+1)}] &=& \frac{\mathcal{T}_{\beta}\psi(s)}{s}\E[I_{\mathcal{T}_{\beta}\psi}^{-s}]
\\
&=&\frac{\psi(s+\beta)}{s+\beta}\E[I_{\mathcal{T}_{\beta}\psi}^{-s}]
\end{eqnarray*}
and
\begin{eqnarray*}
\E[I_{\psi}^{-(s+\beta+1)}] &=& \frac{\psi(s+\beta)}{s+\beta}\E[I_{\psi}^{-(s+\beta+1)}].
\end{eqnarray*}
The first claim follows. Next, we have both $\psi'_{\theta}(0^+)>0$ and $\mathcal{T}_{\delta,\theta}\psi'(0^+)=\frac{\theta-\delta}{\theta}\psi'_{\theta}(0^+)>0$ as $\delta<\theta$. Thus,  the random variables $I_{\psi_{\theta}}$ and $I_{\mathcal{T}_{\delta,\theta}\psi}$ are well defined. Moreover, from \eqref{eq:m_sn}, we get, for any integer $n\geq1$,
\begin{eqnarray*}
\E[I_{\mathcal{T}_{\delta,\theta}\psi}^{-n}] &=&
\frac{\psi'(\theta)(\theta-\delta)}{\theta}\frac{\prod_{k=1}^{n-1}\mathcal{T}_{\delta,\theta}\psi(k)}{\Gamma(n)}\\
&=&
\frac{\psi'(\theta)(\theta-\delta)}{\theta}\frac{\prod_{k=1}^{n-1}\frac{k+\theta-\delta}{k+\theta}\psi_{\theta}(k)}{\Gamma(n)}\\
&=&
\psi'(\theta)\frac{\Gamma(n+\theta-\delta)\Gamma(\theta)}{\Gamma(\theta-\delta)\Gamma(n+\theta)}
\frac{\prod_{k=1}^{n-1}\psi_{\theta}(k)}{\Gamma(n)}.
\end{eqnarray*}
The identity \eqref{eq:f_ex}   follows by moments identification. Then, we use this identity to get, for any $x>0$,
\begin{eqnarray*}
f_{\mathcal{T}_{\delta,\theta}\psi}(x) &=& \frac{\Gamma(\theta)}{\Gamma(\delta)\Gamma(\theta-\delta)}\int_0^1r^{\theta-\delta-1}(1-r)^{\delta-1}f_{\psi_{\theta}}(xr)dr\\ &=& \frac{x^{-\theta}\Gamma(\theta)}{\Gamma(\delta)\Gamma(\theta-\delta)}\int_0^xu^{\theta-\delta-1}(x-u)^{\delta-1}f_{\psi_{\theta}}(u)du \\
&=& \mathbf{D}^{\theta-\delta-1,\delta}f(x).
\end{eqnarray*}
 Next, we  deduce readily from \eqref{eq:rr} that the mapping $s\mapsto \E[I_{\psi_{\theta}}^{-s}]$ is analytic in the right-half plane $\Re(s)>-\theta$. In particular, for any $0<\delta<\theta$, we have $\E[I^{\theta-\delta}_{\psi_{\theta}}] <\infty$. Then, the large asymptotic behavior of the density is obtained by observing that
\begin{eqnarray*}
f_{\mathcal{T}_{\delta,\theta}\psi}(x) &=&  \frac{x^{\delta-\theta-1}\Gamma(\theta)}{\Gamma(\delta)\Gamma(\theta-\delta)}\int_0^xu^{\theta-\delta}(1-u/x)^{\delta-1}f_{\psi_{\theta}}(u)du\\
&\sim&  \frac{x^{\delta-\theta-1}\Gamma(\theta)}{\Gamma(\delta)\Gamma(\theta-\delta)}\int_0^{\infty}u^{\theta-\delta}f_{\psi_{\theta}}(u)du \quad \textrm{ as } x\rightarrow \infty,
\end{eqnarray*}
which completes the proof.
\end{proof}

\begin{exa}[the spectrally negative Lamperti-stable process]\rm
Let us consider the Lamperti-stable process, i.e., for $1<\alpha<2$, $ \psi(u) =
((\alpha-1) (u-1))_{\alpha}, \quad u \geq 0$. Recall that $\psi(1)=0$, $\psi_1(u)=((\alpha-1) u)_{\alpha}$ and that this example is investigated by Patie \cite{Patie-06-poch}.
We get
\begin{eqnarray*}
\E[I^{-n}_{\psi_1}]
&=&
\psi'(0^+)\frac{\Gamma((\alpha-1)n+1)}{\Gamma(\alpha+1)}.
\end{eqnarray*}
Thus, $I_{\psi_1}= \mathbf{e}_1^{-(\alpha-1)}$. Then, for any $0<\delta<1$,
\begin{eqnarray*}
I_{\mathcal{T}_{\delta,\theta}\psi}  &\stackrel{(d)}{=}&
B(1-\delta,\delta)^{-1}\mathbf{e}_1^{-(\alpha-1)},
\end{eqnarray*}
and for any $\beta>0$
\begin{eqnarray*}
I_{\mathcal{T}^{\beta}_{\delta,\theta}\psi}  &\stackrel{(d)}{=}&
B(1+\beta-\delta,\delta)^{-1} G(\beta+1)^{-(\alpha-1)}.
\end{eqnarray*}

\end{exa}
\section{Entrance laws and intertwining relations of pssMp}
In this part, we show that the transformations $\mathcal{T}_{\delta,\beta}$  appear in the study of the entrance law of pssMps. Moreover, also they prove to be useful for the elaboration of intertwining relations between the semigroups of spectrally negative pssMps. We recall that the Markov kernel   $\Lambda$  associated to the positive random variable $V$ is  the multiplicative kernel defined, for a bounded Borelian function $f$, by
\begin{eqnarray*}
\Lambda f(x)=\E[f(Vx)].
\end{eqnarray*}
Then, we say that two Markov semigroups $P_t$ and $Q_t$ are intertwined by the Markov kernel $\Lambda$ if
\begin{eqnarray} \label{eq:def_int}
P_t \Lambda = \Lambda Q_t,\quad t\geq0.
\end{eqnarray}
We refer to the two papers of Carmona et al.~\cite{Carmona-Petit-Yor-94} and  \cite{Carmona-Petit-Yor-98} for a very nice account of intertwining relationships. In particular, they show, by means of the Beta-Gamma algebra, that the semigroup of Bessel processes and the one of the so-called self-similar saw tooth processes are intertwined by the Gamma kernel. Below, we provide alternative examples of such relations for a large class of pssMps with stability index $1$.  Recall that the latter processes were defined in Section \ref{introsection}. We also mention  that for any $\alpha>-1,\delta>0$, the linear operator $\mathbf{D}^{\alpha,\delta}$ defined  in \eqref{eq:defd} is an instance of a Markov kernel which is, in this case, associated to the Beta random variable $B(\alpha+1,\delta)$.
In what follows, when $X$ is associated through the Lamperti mapping to a spectrally negative L\'evy process with Laplace exponent $\psi$, we denote by $P_t^{\psi}$  its corresponding semigroup.  When $\min(\psi(0),\psi'(0^+))<0$ and $\theta<1$, then $P_t^{\psi}$ stands for the semigroup of the unique recurrent extension leaving the boundary point $0$ continuously.  Using the self-similarity property of $X$, we introduce the positive random variable defined, for any bounded borelian function $f$, by
\[ P_t^{\psi}f(0)=\E[f(tJ_{\psi})].\]

Recall that Bertoin and Yor \cite{Bertoin-Yor-02-b} showed that, when $\psi'(0^+)\geq0$, the random variable $J_{\psi}$ is moment-determinate with
\begin{eqnarray} \label{eq:mom_el}
 \E\left[J^{n}_{\psi}\right] &=&  \frac{\prod_{k=1}^{n}\psi(k)}{\Gamma(n+1)},\quad  n=1,2\ldots
\end{eqnarray}

Before stating the new intertwining relations, we provide some further information concerning the entrance law of pssMps. In particular, we show that the expression \eqref{eq:mom_el} of the integer moments  still holds for the entrance law of the unique continuous recurrent extension, i.e.~when $\min\left(\psi(0),\psi'(0^+)\right)<0$ with $\theta<1$. We emphasize that we  consider  both cases when the process $X$ reaches $0$ continuously and by a jump.
\begin{prop}
Let us assume that $\min\left(\psi(0),\psi'(0^+)\right)<0$ with $\theta<1$. Then, we have the following identity in distribution
\begin{equation}
J_{\psi} \stackrel{(d)}{=} B\left(1-\theta,\theta\right)/I_{\mathcal{T}_{1-\theta}\psi_{\theta}}
\end{equation}
where $B(a,b)$  is taken independent of the random variable $I_{\mathcal{T}_{1-\theta}\psi_{\theta}}$.
  Moreover, the entrance law of the unique recurrent extension which  leaves $0$ continuously a.s. is determined by its positive entire moments as follows
\begin{eqnarray*}
 \E\left[J^{n}_{\psi}\right] &=&  \frac{\prod_{k=1}^{n}\psi(k)}{\Gamma(n+1)},\quad  n=1,2\ldots
\end{eqnarray*}
\end{prop}
\begin{proof}
We start by recalling that Rivero \cite[Proposition 3]{Rivero-05} showed that the $q$-potential  of the entrance law of the continuous recurrent extension is given, for a  continuous function $f$, by
\begin{equation*}
\int_0^{\infty}e^{-qt} \E\left[f(tJ_{\psi})\right]dt = \frac{q^{-\theta}}{ C_{\theta}}\int_0^{\infty} f(u) \E\left[e^{-quI_{\psi_{\theta}}}\right] u^{-\theta}du
\end{equation*}
where we have used the self-similarity property of $X$ and set $C_{\theta}=\Gamma(1-\theta)\E\left[ I_{\psi_{\theta}}^{\theta-1}\right]$.
Performing the change of variable $t=uI_{\psi_{\theta}}$ on the right hand side of the previous identity, one gets
\begin{equation*}
\int_0^{\infty}e^{-qt} \E\left[f(tJ_{\psi})\right]dt = \frac{q^{-\theta}}{ C_{\theta}}\int_0^{\infty} e^{-qt}\E\left[f(tI_{\psi_{\theta}}^{-1}) I_{\psi_{\theta}}^{\theta-1}\right]t^{-\theta}dt.
\end{equation*}
Choosing $f(x)=x^{s}$ for some $s \in i\R$ the imaginary line, and using again the self-similarity property, we get
\begin{equation*}
\int_0^{\infty}e^{-qt}t^{s} dt \: \E[J_{\psi}^{s}]= \Gamma\left(s+1\right)q^{-s-1}  \E[J_{\psi}^{s}]
\end{equation*}
where we have used the integral representation of the Gamma function $\Gamma(z)=\int_0^{\infty}e^{-t}t^{z} dt, \Re(z)>-1$.
Moreover, by performing a change of variable, we obtain
\begin{equation}
\int_0^{\infty}t^{s}\E[e^{-qt I_{\psi_{\theta}}}] t^{-1-\theta}dt = q^{-s-\theta-1} \Gamma\left(s-\theta+1\right)\E[I_{\psi_{\theta}}^{-s-\theta-1}].
\end{equation}
Putting the pieces together, we deduce that
\begin{equation} \label{eq:m_el}
 \E[J_{\psi}^{s}]= \frac{\Gamma\left(s-\theta+1\right)}{\Gamma\left(s+1\right)\Gamma\left(1-\frac{\theta}{\alpha}\right)}\frac{\E[I_{\psi_{\theta}}^{-s-\theta-1}]  }{\E[I_{\psi_{\theta}}^{\theta-1}]}
\end{equation}
and the proof of the first claim is completed by moments identification. Next,   we have
\begin{eqnarray*}
 \E\left[J^{n}_{\psi}\right] &=& \frac{ \Gamma(n+1-\theta)}{\Gamma(n+1)\Gamma(1-\theta)\E\left[ I_{\psi_{\theta}}^{\theta-1}\right]} \E\left[ I_{\psi_{\theta}}^{-n+\theta-1}\right] \\
&=& \frac{ \Gamma(n+1-\theta)}{\Gamma(n+1)\Gamma(1-\theta)} \E\left[ I_{\mathcal{T}_{1-\theta}\psi_{\theta}}^{-n}\right] \\
&=& \frac{ \Gamma(n+1-\theta)}{\Gamma(n+1)\Gamma(1-\theta)}\frac{\psi(1)}{1-\theta} \frac{\prod_{k=1}^{n-1}\frac{k}{k+1-\theta}\psi(k+1)}{\Gamma(n)}\\
&=& \frac{ \Gamma(n+1-\theta)\Gamma(n)\Gamma(2-\theta)}{\Gamma(n+1-\theta)\Gamma(n+1)\Gamma(1-\theta)}\frac{\psi(1)}{1-\theta} \frac{\prod_{k=1}^{n-1}\psi(k+1)}{\Gamma(n)}\\
&=&  \frac{\prod_{k=1}^{n}\psi(k)}{\Gamma(n+1)}
\end{eqnarray*}
where we have used, from the second identity, successively the identities \eqref{eq:m_el}, \eqref{*}, \eqref{eq:m_sn} and the recurrence relation of the gamma function. We point out  that under the condition $\theta<1$, $\psi(k)>0$ for any integer $k\geq1$. The proof of the Proposition is then completed.
\end{proof}

Before stating our next result, we recall a criteria given by Carmona et al.~\cite[Proposition 3.2]{Carmona-Petit-Yor-94} for  establishing intertwining relations between pssMps.
If $f$ and $g $ are functions of $C_{0}(\R^+)$, the space of continuous functions vanishing at infinity, satisfying the condition:
\begin{equation} \label{eq:cdi}
\forall t\geq0, \quad P_tf(0)=P_tg(0) \quad  \textrm{then } f=g.
\end{equation}
Then the identity \eqref{eq:def_int} is equivalent to the assertion, for all $f \in C_{0}(\R^+)$,
\[P_1\Lambda f(0) = Q_1f(0).\]
Finally, we introduce the following notation, for any $s\in \mathbb{C},$
\[\mathcal{M}_{\psi}(s)=\E\left[J^{\:is}_{\psi}\right].\]
\begin{thm} \label{thm:int}
\begin{enumerate}
\item Assume that $\psi'(0^+)<0$ and $\mathcal{M}_{\psi_{\theta}}(s)\neq 0$ for any $s\in \R$. Then, for any $\delta<\theta+1$,  we have the following intertwining relationship
\begin{eqnarray*}
P^{\psi_{\theta}}_t \mathbf{D}^{\theta,\delta} = \mathbf{D}^{\theta,\delta} P^{\mathcal{T}_{\delta,\theta}\psi}_t,\quad t\geq0
\end{eqnarray*}
and the following factorization
\begin{eqnarray} \label{eq:id1}
J_{\mathcal{T}_{\delta,\theta}\psi} &\stackrel{(d)}{=}&
B(1+\theta-\delta,\delta) J_{\psi_{\theta}}
\end{eqnarray}
holds.
\item Assume that  $\min\left(\psi(0),\psi'(0^+)\right)<0$ with $\theta<1$ and $\mathcal{M}_{\mathcal{T}_{-\theta}\psi_{\theta}}(s)\neq 0$ for any $s\in \R$. Then, we have the following intertwining relationship
\begin{eqnarray*}
P^{\mathcal{T}_{-\theta}\psi_{\theta}}_t \mathbf{D}^{1,\theta} = \mathbf{D}^{1,\theta} P^{\psi}_t,\quad t\geq0
\end{eqnarray*}
and
\begin{eqnarray*}
J_{\psi} &\stackrel{(d)}{=}&
B(1-\theta,\theta) J_{\mathcal{T}_{-\theta}\psi_{\theta}}.
\end{eqnarray*}
\item Finally, assume that $\psi'(0^+)=0$ and $\mathcal{M}_{\psi}(s)\neq 0$ for any $s\in \R$. Then, for any $\delta<1$,  we have the following intertwining relationship
\begin{eqnarray*}
P^{\psi}_t \mathbf{D}^{1,\delta} = \mathbf{D}^{1,\delta} P^{\mathcal{T}_{\delta,0}\psi}_t,\quad t\geq0
\end{eqnarray*}
and
\begin{eqnarray*}
 J_{\mathcal{T}_{\delta,0}\psi} &\stackrel{(d)}{=}&
B(1-\delta,\delta)J_{\psi}.
\end{eqnarray*}
\end{enumerate}
\end{thm}
\begin{proof}
 First, from the self-similarity property, we observe easily that the condition \eqref{eq:cdi} is equivalent to the requirement that the kernel $M_{\psi_{\theta}}$ associated to the positive random variable $J_{\psi_{\theta}}$ is injective.
% Beside, as the random variable $J_{\psi_{\theta}}$ is moment-determinate is not difficult to verify, from the expressions \eqref{eq:mom_el}, that its Mellin transform satisfies the functional equation
% \[\E\left[J^{s}_{\psi_{\theta}}\right] = \frac{\psi_{\theta}(s)}{s} \E\left[J^{s-1}_{\psi_{\theta}}\right] \]
% which is, at least, valid for any $\Re(s)\geq0$, since the function $\psi_{\theta}$ is well-defined on this domain.
%Since for any $u\in \R$, $\psi_{\theta}(iu)\neq 0$  we then deduce that
%  \[\E\left[J^{iu}_{\psi_{\theta}}\right] \neq 0\]
 Since $\mathcal{M}_{\psi_{\theta}}(s)\neq 0$ for any $s\in\R$, we deduce that the multiplicative kernel $M_{\psi_{\theta}}$ is indeed injective, see e.g.~\cite[Theorem 4.8.4]{Bingham-Goldie-Teugels-89}.
 Next,  note that $\psi'_{\theta}(0^+)>0$ and under the condition $\delta<\theta$, $\mathcal{T}_{\delta,\theta}\psi'(0^+)>0$. Hence, from \eqref{eq:mom_el}, we deduce that, for any $n\geq1$,
\begin{eqnarray*}
\E[J_{\mathcal{T}_{\delta,\theta}\psi}^n] &=&\frac{\Gamma(n+1+\theta-\delta)\Gamma(\theta+1)}{\Gamma(1+\theta-\delta)\Gamma(n+\theta+1)}\E[J_{\psi_{\theta}}^n].
\end{eqnarray*}
The identity \eqref{eq:id1} follows. Both processes being pssMps, the first intertwining relation follows from the criteria given above.
The proof of the Theorem is completed by following similar lines of reasoning for the other claims.
\end{proof}

A nice consequence of the previous result is some interesting relationships between the eigenfunctions of the semigroups of pssMps. Indeed, it is easily seen from the intertwining relation \eqref{eq:def_int} that if a function $f$ is an eigenfunction with eigenvalue $1$ of the semigroup $P_t$ then $\Lambda f$ is an eigenfunction with eigenvalue $1$ of the semigroup $Q_t$.  We proceed by  introducing some notation taken from Patie \cite{Patie-06c}. Set $a_0(\psi)=1$ and define for non-negative integers $n$
\begin{equation*}
a_n(\psi)^{-1}=\prod_{k=1}^n \psi( k).
\end{equation*}
Next, we introduce the entire function $\I_{\psi}$ which admits the series
representation
\begin{equation*}
\I_{\psi}(z)=\sum_{n=0}^{\infty} a_n(\psi) z^{n}, \quad
 z
\in  \C.
\end{equation*}
In \cite[Theorem 1]{Patie-06c}, it is shown that
\begin{equation} \label{eq:eigen}
\mathbf{L}^{\psi} \I_{\psi}(x)=\I_{\psi}(x),\: x>0,
\end{equation}
where, for a smooth function $f$, the linear operator $\mathbf{L}^{\psi}$ is the infinitesimal generator associated to the semigroup $P_t^{\psi}$ and  takes the form
\begin{eqnarray}  \label{eq:inf_ab}
\mathbf{L}^{\psi} f(x) &=&  \frac{\sigma}{2}
xf''(x) + b f'(x)+
x^{-1}\int^{\infty}_{0}f(e^{-r}x)-f(x)+xf'(x)r\mathbb{I}_{\{|r|<1\}}\nu(dr) -\kappa xf(x).
\nonumber
\end{eqnarray}
From the Feller property of the semigroup of $X$, we deduce readily that the identity \eqref{eq:eigen} is equivalent to
\[e^{-t}P_t^{\psi}\I_{\psi}(x)=\I_{\psi}(x), \: t,x\geq0,\]
that is $\I_{\psi}$ is $1$-eigenfunction for $P_t^{\psi}$.  Hence, we deduce from Theorem \ref{thm:int} the following.
\begin{cor}
\begin{enumerate}
\item Let $\psi'(0^+)<0$. Then, for any $\delta<\theta+1$,  we have the following identity
\begin{eqnarray*}
 \mathbf{D}^{\theta,\delta} \I_{\psi_{\theta}}(x)= \I_{\mathcal{T}_{\delta,\theta}\psi}(x).
\end{eqnarray*}
\item If $\psi'(0^+)<0$ and $\theta<1$, then
\begin{eqnarray*}
 \mathbf{D}^{1,\theta} \I_{\mathcal{T}_{-\theta}\psi_{\theta}}(x)= \I_{\psi}(x).
\end{eqnarray*}
\item Finally, if $\psi'(0^+)=0$, then, for any $\delta<1$,  we have
\begin{eqnarray*}
 \mathbf{D}^{1,\delta} \I_{\psi}(x)= \I_{\mathcal{T}_{\delta,0}\psi }(x).
\end{eqnarray*}
\end{enumerate}
\end{cor}
We illustrate this last result by detailing some interesting instances of such relationships between some known special functions.
\begin{exa}[Mittag-Leffler type functions]\rm
Let us consider, for any $1<\alpha<2$, the Laplace exponent $\psi(u) = (\alpha(u-1/\alpha))_\alpha$. We easily check that $\theta = 1/\alpha$ and  we have $ \psi_{1/\alpha}(u) =
(\alpha u)_{\alpha}$. Observing that
\[\prod_{k=1}^n \psi_{1/\alpha}(k)=\prod_{k=1}^n (\alpha k)_{\alpha}=\frac{\Gamma\left(\alpha(n+1)\right)}{\Gamma(\alpha)},\: n\geq1, \]
and using the fact that the random variable $J_{\psi_{\theta}}$ is moment-determinate, we readily check, from the expression \eqref{eq:mom_el}, that
\[ \mathcal{M}_{\psi_{\theta}}(s)= \frac{\Gamma(\alpha(is+1))}{\Gamma(\alpha) \Gamma(is+1)} \]
The pole of the gamma function being the negative integers, the condition $\mathcal{M}_{\psi_{\theta}}(s)\neq 0$  is satisfied for any $s\in \R$. Moreover, we obtain
\begin{equation*}
\I_{\psi_{1/\alpha}}(x)= \Gamma(\alpha)\mathcal{E}_{\alpha,\alpha}(x)
\end{equation*}
where we recall that the Mittag-Leffler function $\mathcal{E}_{\alpha,\alpha}$ is defined in the example \ref{ex:sp}. Next, for any $\delta<1+1/\alpha$, we have
\begin{eqnarray*}
\I_{\mathcal{T}_{\delta,1/\alpha}\psi}(x)&=&  \frac{\Gamma(\alpha)\Gamma(1/\alpha+1-\delta)}{\Gamma(1/\alpha+1)}\sum_{n=0}^{\infty} \frac{\Gamma(n+1/\alpha+1)}{\Gamma(n+1/\alpha+1-\delta)\Gamma(\alpha n +\alpha)} x^{ n} \\
  &=&
{}_2F_2\left( \left.\begin{array}{c}
                  (1,1/\alpha+1),\left(1,1
\right) \nonumber \\
                  (1,1/\alpha+1-\delta), (\alpha,\alpha)
                \end{array} \right|
  x\right)
,
\end{eqnarray*}
where  ${}_2F_2$ is the Wright hypergeometric function, see e.g.~Braaksma \cite[Chap.
12]{Braaksma-64}.
Hence, we have
\begin{equation*}
\Gamma(\alpha)\mathbf{D}^{\theta,\delta} \mathcal{E}_{\alpha,\alpha}(x)= {}_2F_2\left( \left.\begin{array}{c}
                  (1,1/\alpha+1),\left(1,1
\right) \nonumber \\
                  (1,1/\alpha+1-\delta), (\alpha,\alpha)
                \end{array} \right|
  x\right).
\end{equation*}
\end{exa}

\begin{exa}\rm
Now, for any $1<\alpha<2$, we set $ \psi(u) =
u^{\alpha}$ and we  note that $\psi'(0^+)=0$. Proceeding as in the previous example, we get
\[ \mathcal{M}_{\psi}(s)= \Gamma^{\alpha-1}(is+1) \]
 and hence the condition $\mathcal{M}_{\psi}(s)\neq 0$  is satisfied for any $s\in \R$.   We have
\begin{equation*}
\I_{\psi}(x)= \sum_{n=0}^{\infty}\frac{1}{\Gamma^{\alpha}(n+1)}x^n
\end{equation*}
and, for any $\delta<1$, we write
\begin{equation*}
\I_{\mathcal{T}_{\delta,0}\psi}(x)= \Gamma(1-\delta)\sum_{n=0}^{\infty}\frac{\Gamma(n+1)}{\Gamma(n+1-\delta)\Gamma^{\alpha}(n+1)}x^n.
\end{equation*}
Consequently,
\begin{equation*}
\mathbf{D}^{1,\delta}\I_{\psi}(x) = \I_{\mathcal{T}_{\delta,0}\psi}(x).
\end{equation*}
\end{exa}

%\bibliography{./../../../Bibliography/bib_pp}
\end{document}